\documentclass{article}

\usepackage{PRIMEarxiv}

\usepackage[utf8]{inputenc} % allow utf-8 input
\usepackage[T1]{fontenc}    % use 8-bit T1 fonts
\usepackage{hyperref}       % hyperlinks
\usepackage{url}            % simple URL typesetting
\usepackage{booktabs}       % professional-quality tables
\usepackage{amsfonts,amsmath,amssymb,amsthm}       % blackboard math symbols
\usepackage{nicefrac}       % compact symbols for 1/2, etc.
\usepackage{microtype}      % microtypography
\usepackage{lipsum}
\usepackage[square,numbers]{natbib}
\usepackage{fancyhdr}       % header
\usepackage{graphicx}       % graphics
\usepackage{float}
\usepackage{subcaption}
\graphicspath{{media/}}     % organize your images and other figures under media/ folder
\usepackage{bbold}
\usepackage{orcidlink}

\usepackage{tikz}
\usetikzlibrary{graphs, positioning}
\usetikzlibrary{calc}

\newtheorem{theorem}{Theorem}[section]

\newtheorem{corollary}[theorem]{Corollary}
\newtheorem{proposition}[theorem]{Proposition}

\newcommand{\bigzero}{\mbox{\normalfont\bfseries 0}}
\newcommand{\bigjota}{\mbox{\normalfont\bfseries J}}
\title{Eigenvalues of $\boldsymbol{L_\alpha}-$matrices under graph operations
}

\author{
    Gabriel Roberto Silva de Lima\\
    Mathematics Department \\
    Escola Nacional de Ciências Estatísticas, ENCE/IBGE \\
    Rio de Janeiro, Brazil\\
     \texttt{gabroberto2@gmail.com}\\
    \And
   Carla Silva Oliveira\orcidlink{0000-0001-6684-8811} \\
  Mathematics Department \\
  Escola Nacional de Ciências Estatísticas, ENCE/IBGE \\
  Rio de Janeiro, Brazil\\
  \texttt{carla.oliveira@ibge.gov.br} \\
   \And
  João Domingos Gomes da Silva Junior\orcidlink{0000-0002-1745-0302}\\
  Mathematics Department \\
  Colégio Pedro II \\
  Rio de Janeiro, Brazil \\
  \texttt{joao.dgomes@gmail.com} \\
}

\begin{document}
\maketitle

\begin{abstract}
    Let $G$ be a simple graph, $A(G)$ its adjacency matrix, and $D(G)$ its diagonal degree matrix. In 2022, \citeauthor{Wang2020} (\cite{Wang2020}) defined the family of matrices $L_\alpha$ as the convex linear combination:
    \[
    L_\alpha(G) = \alpha D(G) + (\alpha - 1)A(G),
    \]
    where $\alpha \in [0,1]$. 
    The study of the spectrum of this family of matrices may provide a unified framework for analyzing 
    the spectra of the adjacency, degree, and Laplacian matrices ($D(G) - A(G)$). 
    In this work, we investigate the spectrum of $L_\alpha$ under graph operations 
    and within specific families of graphs.
\end{abstract}

% keywords can be removed

\section{Introduction}\label{sec1}
Let $G=(V(G),E(G))$ or simply $G=(V,E)$ be a simple graph. The cardinality $|V|$ is the \textit{order} of $G$, and $|E|$ is its \textit{size}. Two vertices $u,v$ are adjacent iff $\{u,v\}\in E$. An independent set is a subset of $V$ with no adjacent vertices. The open neighborhood of \( v \) is \( N_G(v) \subseteq V \); the closed neighborhood is \( N_G[v]=N_G(v)\cup\{v\} \). The degree of \( v_i \in V \), denoted \( d_G(v_i) \) or \( d(v_i) \), is equal to $|N_G(v_i)|$ and $\displaystyle \sum_{i=1}^{n} d(v_i)=2m.$ An edge is called \textit{pendant} if it is incident to a vertex of degree $1$. The minimum degree is $\delta(G)=\delta=\min\{d(v):v\in V\}$ and the maximum degree is $\Delta(G)=\Delta=\max\{d(v):v\in V\}$. Two vertices $u,v\in V$ are \emph{true twins} if \( N_G[u]=N_G[v] \), and \emph{false twins} if they are nonadjacent and \( N_G(u)=N_G(v) \). A graph is \emph{complete}, denoted $K_n$, if every pair of distinct vertices is adjacent. A graph $G$ is \emph{regular} if all vertices have the same degree; we say $G$ is $r$-regular when the common degree is $r$. $G$ is \emph{bipartite} if $V$ can be partitioned into two independent parts $V_1,V_2$ with $V_1\cap V_2=\emptyset$ and $E\subseteq \{\{u,v\}:u\in V_1, v\in V_2\}$. If every vertex of $V_1$ (with $|V_1|=p$) is adjacent to every vertex of $V_2$ (with $|V_2|=q$), then $G$ is the complete bipartite graph  denoted by $K_{p,q}$. The path on $n$ vertices, $P_n$, is the simple graph with $V=\{v_1,\dots,v_n\}$ and $E=\{\{v_i,v_{i+1}\}\mid 1\le i<n\}$ and, for $n\ge 3$, the cycle $C_n$ is the simple graph with $V=\{v_1,\dots,v_n\}$ and $E=\{\{v_i,v_{i+1}\}\mid 1\le i<n\}\cup \{\{v_n,v_1\}\}$.

Let $G=(V_1,E_1)$ and $H=(V_2,E_2)$ be  simple graphs. In the case, $V_1\cap V_2=\emptyset$, the union between $G$ and $H$ is defined by $G\cup H=(V_1\cup V_2,\ E_1\cup E_2)$. If $V_1\cap V_2=\emptyset$, the join $G\lor H$ is defined by $V(G\lor H)=V_1\cup V_2$ and $E(G\lor H)=E_1\cup E_2\cup \{\{u,v\}: u\in V_1,\, v\in V_2\}$.The cartesian product $G\times H$ has vertex set $V_1\times V_2$, where two vertices $(v_1,v_2)$ and $(v_1',v_2')$ are adjacent if and only if either $v_1=v_1'$ and $v_2$ is adjacent to $v_2'$ in $H$, or $v_2=v_2'$ and $v_1$ is adjacent to $v_1'$ in $G$. The direct product $G\odot H$ has vertices $(v_1,v_2)$ adjacent to $(v_1',v_2')$ if and only if $v_1$ is adjacent to $v_1'$ in $G$ and $v_2$ is adjacent to $v_2'$ in $H$. The strong product $G\otimes H$ connects $(v_1,v_2)$ and $(v_1',v_2')$ if and only if $v_1=v_1'$ or $v_1\sim v_1'$ in $G$, and simultaneously $v_2=v_2'$ or $v_2\sim v_2'$ in $H$. The coalescence of $G$ and $H$, denoted by $G \cdot H$, is a graph of order $\vert V_1 \vert + \vert V_2 \vert - 1$ that can be obtained from the graph $G \cup H$ by identifying some vertex of $G$ with some vertex of $H$. Given a graph $G$, the splitting graph $S(G)$ is obtained by adding, for each $v\in V_1$, a new vertex $v'$ adjacent to every vertex in $N_G(v)$. These operations defined above are illustrated in Figure~\ref{fig:operation} for $G\simeq P_2$ and $H \simeq P_3$.

 For $p, q$ natural numbers such that $q\ge 1$ and $p\ge 3$, the pineapple graph $K_p^{q}$ is obtained by attaching $q$ pendant vertices to a vertex of $K_p$. For $n\ge 3$ and $1\le l\le n$, an integer number, the graph $H_n^l$ is constructed from two copies of $K_n$ by adding $l$ edges between $l$ vertices of one copy and their corresponding vertices in the other. Under the same conditions, the graph $KK_n^l$ is obtained from two copies of $K_n$ by adding $l$ edges between a vertex of one copy and $l$ vertices of the other. Finally, for $c \ge 1$, $s \ge 1$, and $\eta \ge 2$, where $c, s, \eta$ are natural number, the Core–Satellite graph is defined by $\Theta(c,s,\eta)=K_c \lor \eta K_s$, where $\eta K_s$ denotes the disjoint union of $\eta$ copies of $K_s$. Some examples of these graph families are illustrated in Figure~\ref{fig:graphs}.

\begin{figure}[H]
    \centering
    \includegraphics[width=0.9\linewidth]{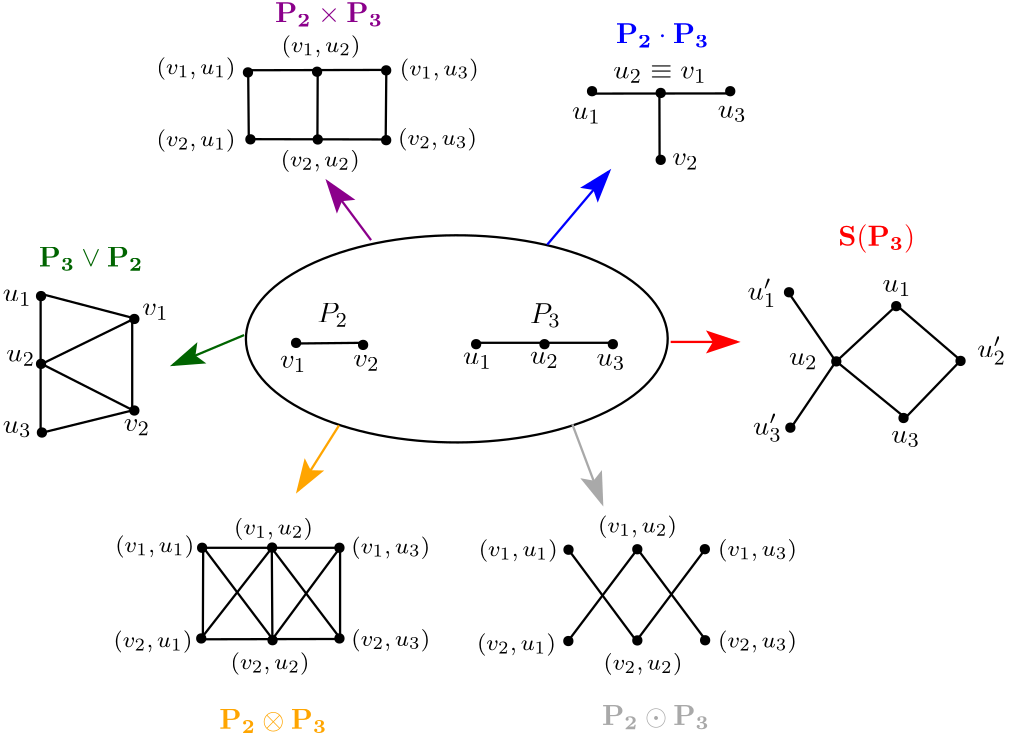}
    \caption{Examples of the graph operations and the graph $S(P_3)$} 
    \label{fig:operation}
\end{figure}

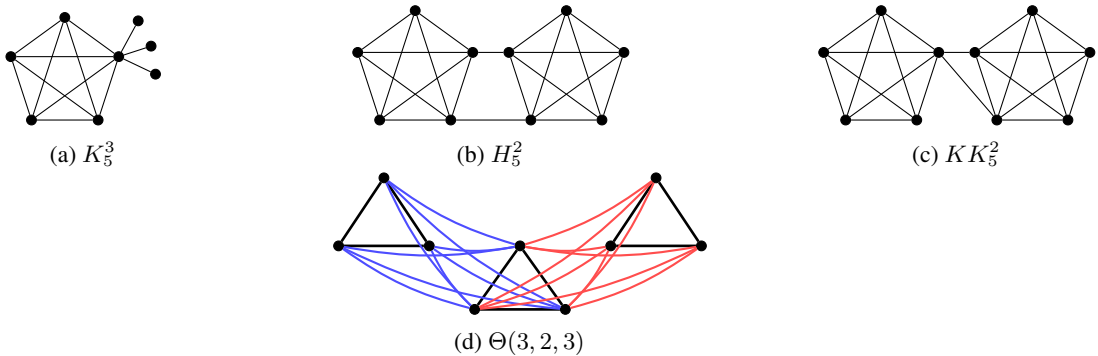
\begin{figure}[H]
\centering
\begin{subfigure}{0.3\textwidth}
\centering
\begin{tikzpicture}[scale=0.75, 
    vertex/.style={circle, draw, fill=black, inner sep=1.3pt}]

    \node[vertex] (a) at (90:1) {};
    \node[vertex] (b) at (162:1) {};
    \node[vertex] (c) at (234:1) {};
    \node[vertex] (d) at (306:1) {};
    \node[vertex] (e) at (18:1) {};

    \node[vertex] (f) at (36:1.6) {};
    \node[vertex] (g) at (18:1.6) {};
    \node[vertex] (h) at (0:1.6) {};

    \foreach \i/\j in {a/b, a/c, a/d, b/c, b/d, c/d}
        \draw (\i)--(\j);

    \foreach \i in {a,b,c,d}
        \draw (e)--(\i);

    \foreach \i in {f,g,h}
        \draw (e)--(\i);

\end{tikzpicture}
\caption{$K_5^3$}
\end{subfigure}
\quad
\begin{subfigure}{0.3\textwidth}
\centering
\begin{tikzpicture}[scale=0.8, 
    vertex/.style={circle, draw, fill=black, inner sep=1.3pt}]

    \node[vertex] (v1) at (90:1) {};
    \node[vertex] (v2) at (162:1) {};
    \node[vertex] (v3) at (234:1) {};
    \node[vertex] (v4) at (306:1) {};
    \node[vertex] (v5) at (18:1) {};

    \foreach \i/\j in {v1/v2, v1/v3, v1/v4, v1/v5, v2/v3, v2/v4, v2/v5, v3/v4, v3/v5, v4/v5}
        \draw (\i)--(\j);

    \def\shift{2.5}
    \node[vertex] (w1) at ($(90:1)+(\shift,0)$) {};
    \node[vertex] (w2) at ($(162:1)+(\shift,0)$) {};
    \node[vertex] (w3) at ($(234:1)+(\shift,0)$) {};
    \node[vertex] (w4) at ($(306:1)+(\shift,0)$) {};
    \node[vertex] (w5) at ($(18:1)+(\shift,0)$) {};

    \foreach \i/\j in {w1/w2, w1/w3, w1/w4, w1/w5, w2/w3, w2/w4, w2/w5, w3/w4, w3/w5, w4/w5}
        \draw (\i)--(\j);

    \draw (v5)--(w2);
    \draw (v4)--(w3);

\end{tikzpicture}
\caption{$H_5^2$}
\end{subfigure}
\hfill
\begin{subfigure}{0.3\textwidth}
\centering
\begin{tikzpicture}[scale=0.8, 
    vertex/.style={circle, draw, fill=black, inner sep=1.3pt}]

    \node[vertex] (v1) at (90:1) {};
    \node[vertex] (v2) at (162:1) {};
    \node[vertex] (v3) at (234:1) {};
    \node[vertex] (v4) at (306:1) {};
    \node[vertex] (v5) at (18:1) {};

    \foreach \i/\j in {v1/v2, v1/v3, v1/v4, v1/v5, v2/v3, v2/v4, v2/v5, v3/v4, v3/v5, v4/v5}
        \draw (\i)--(\j);

    \def\shift{2.5}
    \node[vertex] (w1) at ($(90:1)+(\shift,0)$) {};
    \node[vertex] (w2) at ($(162:1)+(\shift,0)$) {};
    \node[vertex] (w3) at ($(234:1)+(\shift,0)$) {};
    \node[vertex] (w4) at ($(306:1)+(\shift,0)$) {};
    \node[vertex] (w5) at ($(18:1)+(\shift,0)$) {};

    \foreach \i/\j in {w1/w2, w1/w3, w1/w4, w1/w5, w2/w3, w2/w4, w2/w5, w3/w4, w3/w5, w4/w5}
        \draw (\i)--(\j);

    \draw (v5)--(w2);
    \draw (v5)--(w3);

\end{tikzpicture}
\caption{$KK_5^2$}
\end{subfigure}

\begin{subfigure}{0.7\textwidth}
\centering
\begin{tikzpicture}[
    scale=0.6,
    vertex/.style={circle, draw, fill=black, inner sep=1.3pt},
    internaledge/.style={black, line width=1pt},
    joinL/.style={blue!70, line width=0.8pt},
    joinR/.style={red!70, line width=0.8pt}
]

%-------------------------
% K3 central
%-------------------------
\node[vertex] (v1) at (0,0) {};
\node[vertex] (v2) at (-1,-1.4) {};
\node[vertex] (v3) at (1,-1.4) {};
\draw[internaledge] (v1)--(v2)--(v3)--(v1);

%-------------------------
% K3 esquerdo
%-------------------------
\node[vertex] (u1) at (-3,1.5) {};
\node[vertex] (u2) at (-4,0) {};
\node[vertex] (u3) at (-2,0) {};
\draw[internaledge] (u1)--(u2)--(u3)--(u1);

%-------------------------
% K3 direito
%-------------------------
\node[vertex] (w1) at (3,1.5) {};
\node[vertex] (w2) at (4,0) {};
\node[vertex] (w3) at (2,0) {};
\draw[internaledge] (w1)--(w2)--(w3)--(w1);

%-------------------------
% Join esquerdo (AZUL)
%-------------------------
\foreach \v in {v1,v2,v3}{
  \foreach \u in {u1,u2,u3}{
    \draw[joinL] (\v) to[bend left=10] (\u);
  }
}

%-------------------------
% Join direito (VERMELHO)
%-------------------------
\foreach \v in {v1,v2,v3}{
  \foreach \w in {w1,w2,w3}{
    \draw[joinR] (\v) to[bend right=10] (\w);
  }
}

\end{tikzpicture}
\caption{$\Theta(3,2,3)$}
\end{subfigure}

\caption{Examples of some families of graphs.}
\label{fig:graphs}
\end{figure}

The adjacency matrix $A(G)=[a_{ij}]$ is the $n \times n$ matrix with entries $a_{ij}=1$ if $\{v_i,v_j\}\in E(G)$ and $a_{ij}=0,$ otherwise;
and the degree matrix $D(G)=[d_{ij}]$ is defined by $ d_{ij}= d(v_i)$,  if  $i=j,$ and $0$, otherwise. The diagonal matrix $D(G)$ can be written as $D(G)=\mathrm{diag}(d(v_1), d(v_2), \ldots, d(v_n))$. The Laplacian and signless Laplacian matrices of $G$ are $L(G)=D(G)-A(G)$ and $Q(G)=D(G)+A(G)$, respectively, and all $L(G)-$eigenvalues are nonnegative, with the smallest is equal to $0$. 
Denoted by $j_n$ is a vector whose entries are all equal to $1$. This vector is an eigenvector of $L(G)$ associated with the eigenvalue zero.

In recent years, increasing attention has been given to convex linear combinations of graph-associated matrices, which provide a unified for to study properties of some spectral matrices at the same time. These aforementioned matrices can be seen at \cite{Wang2020, Nikiforov2017, CUI20191, Samanta2024}.

In 2017, Nikiforov (\cite{Nikiforov2017}) introduced the first convex linear combination, $A_{\alpha},$ defined the following way:
\[
A_\alpha(G) = \alpha D(G) + (1-\alpha)A(G),
\]
for all $\alpha \in [0,1]$. 

\citeauthor{Wang2020} (\cite{Wang2020}) introduced the other linear convex combination
\[
L_\alpha(G) = \alpha D(G) + (\alpha - 1)A(G), \quad \forall \alpha \in [0,1],
\]
which includes the Laplacian matrix that was not in the convex linear combination $A_{\alpha}$. This family contains $A(G)$, $D(G)$, and $L(G)$; thus, studying $L_\alpha(G)$ unifies the theory of these matrices.

In this paper, we investigate the spectral properties of  $L_\alpha(G)$ under various graphs operations and across different families of graphs. The paper is organized as follows: Section~\ref{sec2} presents the definitions and results required for the development of this work. In Section~\ref{sec3} the main results are presented and in Section~\ref{sec4} our conclusions and future works.

\section{Preliminaries}\label{sec2}

Let \( M_{m,n}(\mathbb{R}) \), or simply \( M_{m,n} \), denote the vector space of real \( m \times n \) matrices, and write \( M_n \) when \( m=n \). We denote by \( I_n \) the identity matrix, \( \bigjota_{m,n} \) the all-ones matrix, \( \bigzero_{m,n} \) the zero matrix, and \( j_n \in \mathbb{R}^n \) the all-ones vector. The roots of the characteristic polynomial of $M,$ $P_M(x),$ are called the $M-$eigenvalues,
and denoted by \( \lambda_1(M) \ge \cdots \ge \lambda_n(M) \). The algebraic multiplicity of \( \lambda_i(M) \) is denoted by \( s_i \), and the $M-$spectrum, the set of eigenvalues of $M$, is denoted by \( \mathrm{Spec}(M) \). The spectral radius is $\rho(M)=\max\{|\lambda|:\ \lambda \in \mathrm{Spec}(M)\}.$ The following result is a standard consequence of the fact that a nonzero polynomial has only finitely many roots.

\begin{proposition}[\cite{lang2002algebra, dummit2004abstract}]\label{prop:identidadepolinomial}
Let $p(\lambda)$ and $q(\lambda)$ be polynomials with real coefficients. Suppose there exists an infinite set $A \subset \mathbb{R}$ such that $p(\lambda)=q(\lambda), \quad \forall\, \lambda \in A.$ Then
\[
p(\lambda)\equiv q(\lambda),
\]
that is, $p$ and $q$ are identical polynomials.
\end{proposition}

\begin{corollary}[\cite{lang2002algebra, dummit2004abstract}]\label{cor:fora_do_finito}
Let $p(\lambda)$ and $q(\lambda)$ be polynomials. If $p(\lambda)=q(\lambda)$ for all $\lambda$ outside a finite set, then
\[
p(\lambda)\equiv q(\lambda).
\]
\end{corollary}

For \( M \in M_{m,n} \) and index sets \( \alpha \subseteq \{1,\ldots,m\} \), \( \beta \subseteq \{1,\ldots,n\} \), we denote by \( M[\alpha,\beta] \) the submatrix formed by the rows indexed by \( \alpha \) and the columns indexed by \( \beta \). A matrix \( M \in M_n \) is said to be partitioned into blocks if it can be written in block form as $M = (B_{ij}),$
where each \( B_{ij} \) is a submatrix of \( M \) called ``block''. The following result involves the spectrum of a block matrix.

\begin{theorem}[\cite{horn_johnson_2013}] \label{theo:union}
Let \( P = \begin{bmatrix} A & B \\ 0 & C \end{bmatrix} \) be a block matrix of order \( n \), where \( A \) and \( C \) are square. Then
\[
\mathrm{Spec}(P) = \mathrm{Spec}(A) \cup \mathrm{Spec}(C).
\]
\end{theorem}

Theorem~\ref{equitable} is a fundamental tool in Spectral Graph Theory, as it enables a significant reduction of spectral complexity. In particular, by Theorem~\ref{quociente}, the spectrum of the quotient matrix captures essential eigenvalue information of the original matrix.

\begin{theorem}[\cite{horn_johnson_2013}]\label{equitable}

Let \( A \) be a symmetric $n\times n$ matrix written in block form
\[
A = 
\begin{bmatrix}
A_{1,1} & A_{1,2} & \cdots & A_{1,k} \\
A_{2,1} & A_{2,2} & \cdots & A_{2,k} \\
\vdots & \vdots & \ddots & \vdots \\
A_{k,1} & A_{k,2} & \cdots & A_{k,k}
\end{bmatrix},
\]
where \( A_{i,j} \) is of order \( n_i \times n_j \) and each of its rows has constant sum \( c_{i,j} \). If \(M = [c_{i,j}]_{k,k}\), then the eigenvalues of \( M \) are also eigenvalues of \( A \). The matrix $M$ is called the quotient matrix of $A$.
\end{theorem}

\begin{theorem}[\cite{YOU2019}] \label{quociente}
Let $M$ be a quotient matrix of $A$ as in Theorem~\ref{equitable}. Then
\begin{itemize}
\item[(i)] $\mathrm{Spec}(M) \subseteq \mathrm{Spec}(A)$;
\item[(ii)] The largest eigenvalue of $M$ equals the largest eigenvalue of $A$ if $A$ is nonnegative (i.e., all entries of $A$ are nonnegative).
\end{itemize}
\end{theorem}

The following results concern the determinant of block matrices.

\begin{theorem}[\cite{horn_johnson_2013}] \label{determinante}Let \( A, B, C, D \in \mathbb{M}_n \) be block matrices . Then
    \[
    \det
    \begin{bmatrix}
    A & B \\
    C & D
    \end{bmatrix} =
    \begin{cases}
    \det(AD - CB), & \text{se } AC = CA; \\
    \det(DA - CB), & \text{se } AB = BA; \\
    \det(AD - BC), & \text{se } DC = CD; \\
    \det(DA - BC), & \text{se } DB = BD.
    \end{cases}
    \]
\end{theorem}

\begin{proposition}[\cite{horn_johnson_2013}]\label{prop:schur-scalar}
Let
\[
M=
\begin{bmatrix}
a & x^T\\
x & B
\end{bmatrix},
\]
where $a\in\mathbb{R}$, $x$ is a column vector, and $B$ is an invertible square matrix. Then
\[
\det(M)=\det(B)\bigl(a-x^TB^{-1}x\bigr).
\]
\end{proposition}

\begin{corollary}[\cite{horn_johnson_2013}]\label{cor:schur-blockdiag}
Let $X$ and $Y$ be invertible square matrices, $\nu$ and $\rho$ be compatible column vectors, and let $a\in\mathbb{R}$. Then
\[
\det
\begin{bmatrix}
a & \nu^T & \rho^T\\
\nu & X & 0\\
\rho & 0 & Y
\end{bmatrix}
=
\det(X)\det(Y)\Bigl(a-\nu^TX^{-1}\nu-\rho^TY^{-1}\rho\Bigr).
\]
\end{corollary}

Let \( A = [a_{ij}] \in M_{m,n} \) and \( B = [b_{ij}] \in M_{p,q} \). The Kronecker product of \( A \) and \( B \), denoted \( A \otimes B \), is the \(mp \times nq\) matrix defined by
\[
A \otimes B =
\begin{bmatrix}
a_{11}B & a_{12}B & \cdots & a_{1n}B \\
a_{21}B & a_{22}B & \cdots & a_{2n}B \\
\vdots & \vdots & \ddots & \vdots \\
a_{m1}B & a_{m2}B & \cdots & a_{mn}B
\end{bmatrix}.
\]
From Proposition~\ref{prop_Kronecker}, we obtain the following properties of this product.

\begin{proposition}
[\cite{broxson_2006}]\label{prop_Kronecker}
Let \(A,B,C, D \in M_{n}\) and \(\lambda \in \mathbb{R}\). Then:
\begin{enumerate}
\item[(1)] \((\lambda A) \otimes B = \lambda (A \otimes B) = A \otimes (\lambda B)\);
\item[(2)] \((A \otimes B)^{T} = A^{T} \otimes B^{T}\);
\item[(3)] \((A \otimes B) \otimes C = A \otimes (B \otimes C)\);
\item[(4)] \((A + B) \otimes C = A \otimes C + B \otimes C\);
\item[(5)] \(A \otimes (B + C) = A \otimes B + A \otimes C\);
\item[(6)] \(A \otimes \mathbf{0}_{n} = \mathbf{0}_{n} \otimes A = \mathbf{0}_{n}\);
\item[(7)] \(I_{n} \otimes I_{m} = I_{nm}\);
\item[(8)] \( (A\otimes B)(C\otimes D) = (AC) \otimes (BD)\).
\end{enumerate}
\end{proposition}

The literature on $L_\alpha(G)$ is still relatively limited; existing works focus mainly on eigenvalue analysis, bounds, and properties for specific classes (e.g., regular graphs as can be seen in \cite{Wang2020, Fakieh2024, Carla2022}).

\begin{theorem}[\cite{Carla2022}]\label{theo:twins}
Let $G$ be a graph of order at least $2$ and let $v_i$, $v_{j_p}$, $1\le p\le n$, be twin vertices.
\begin{itemize}
\item[(1)] If $v_i$ and $v_{j_p}$ are true twins, then $\alpha(d(v_i)-1)+1$ is an eigenvalue of $L_\alpha(G)$.
\item[(2)] If $v_i$ and $v_{j_p}$ are false twins, then $\alpha d(v_i)$ is an eigenvalue of $L_\alpha(G)$.
\end{itemize}
\end{theorem}

\begin{theorem}
[\cite{Carla2022}]
If $G$ is a $k$-regular graph of order $n$ and $\alpha\in[0,1]$, then for each $1\le i\le n$,
\[
\lambda_i(L_\alpha(G))=\alpha k+(\alpha-1)\lambda_i(A(G)).
\]
\end{theorem}

\begin{theorem}[\cite{Carla2022}]\label{complete}
For each $\alpha\in[0,1]$, the $L_\alpha$-spectrum of $K_n$ is
\[
\mathrm{Spec}(L_\alpha(K_n))=\{\,2\alpha n-2\alpha-n+1,\ (\alpha n-2\alpha+1)^{(n-1)}\,\}.
\]
\end{theorem}

\begin{theorem}[\cite{Wang2020}]
A graph $G$ is bipartite if and only if $L_\alpha(G)$ and $A_\alpha(G)$ have the same spectrum, where $A_\alpha(G)=\alpha D(G)+(1-\alpha)A(G)$ for $\alpha\in[0,1]$.
\end{theorem}

\begin{theorem}[\cite{Carla2022, Nikiforov2017}]
\label{car22}
Let $p\ge q\ge 1$ and $\alpha\in[0,1]$. Then:
\begin{enumerate}
\item[(1)] $\displaystyle \lambda_{1}(L_{\alpha}(K_{p,q}))=\frac{1}{2}\Big[\alpha(p+q)+\sqrt{\alpha^{2}(p+q)^{2}+4pq(1-2\alpha)}\Big]$;
\item[(2)] $\displaystyle \lambda_{n}(L_{\alpha}(K_{p,q}))=\frac{1}{2}\Big[\alpha(p+q)-\sqrt{\alpha^{2}(p+q)^{2}+4pq(1-2\alpha)}\Big]$;
\item[(3)] $\lambda_{k}(L_{\alpha}(K_{p,q}))=\alpha p$ for $1<k\le q$;
\item[(4)] $\lambda_{k}(L_{\alpha}(K_{p,q}))=\alpha q$ for $q<k\le p+q$.
\end{enumerate}
\end{theorem}

As consequence of Theorem \ref{car22} has the following result.

\begin{corollary} [\cite{Fakieh2024}]
The $L_\alpha$-spectrum of the star $K_{1,n-1}$ is:
\begin{align*}
\lambda({L_\alpha}(K_{1,n-1})) &= \tfrac{1}{2}\Big(\alpha n+\sqrt{\alpha^{2}n^{2}+4(n-1)(1-2\alpha)} \Big), \\
\lambda({L_\alpha}(K_{1,n-1})) &= \tfrac{1}{2}\Big(\alpha n-\sqrt{\alpha^{2}n^{2}+4(n-1)(1-2\alpha)} \Big), \\
\lambda_{k}({L_\alpha}(K_{1,n-1})) &= \alpha \quad \text{for } 1 < k < n.
\end{align*}
\end{corollary}

\begin{theorem}[\cite{Wang2020}]
If $G$ is a connected graph of order $n$ and $\alpha \ge \tfrac{1}{2}$, then all eigenvalues of $L_\alpha(G)$ are nonnegative.
\end{theorem}

\section{Main results}\label{sec3}

This section presents results of the spectrum for the family of matrices $L_\alpha$ on graphs obtained via union, join, graph products, splitting, and coalescence of vertices. The spectra of the Pineapple, $H_n^l$, $KK_n^l$, and Core–Satellite graph are determined.

%\subsection{Union of graphs}

\subsection{Results for graph operations}

%Label the vertices of $G \cup H$ so that those of $G$ come first and those of $H$ come last. Then

Throughout this subsection, we study the spectrum of $L_{\alpha}$-matrices of graphs obtained from operations on $G$ and $H$, where $G$ and $H$ are graphs of order $n_1$ and $n_2$, respectively. To obtain the $L_{\alpha}-$matrices, the vertices are labeled so that those of $G$ come first, followed by those of $H$.

For the union $G \cup H$, we have
\begin{align*}
A(G \cup H) &= \begin{bmatrix} 
A(G) & 0 \\ 
0 & A(H) 
\end{bmatrix}, \qquad
D(G \cup H) = \begin{bmatrix} 
D(G) & 0 \\ 
0 & D(H) 
\end{bmatrix}.
\end{align*}
Hence, for $\alpha \in [0,1]$,
\begin{equation}\label{matriz_L_alpha_uniao}
L_\alpha(G \cup H) =
\begin{bmatrix} 
L_\alpha(G) & 0 \\ 
0 & L_\alpha(H) 
\end{bmatrix}.
\end{equation}

\begin{theorem}\label{theo:spec_union}
Let $\alpha \in [0,1]$. Then
\[
\mathrm{Spec}(L_\alpha(G \cup H)) 
= \mathrm{Spec}(L_\alpha(G)) \cup \mathrm{Spec}(L_\alpha(H)).
\]
\begin{proof}
With the vertex ordering described above, the matrix $L_\alpha(G \cup H)$ has the block diagonal form given in~\eqref{matriz_L_alpha_uniao}. From Theorem~\ref{theo:union}, the result follows immediately from the spectral properties of block diagonal matrices.
\end{proof}
\end{theorem}

For the operation join between two graphs, $G \lor H$, we have
\begin{align*}
\resizebox{\textwidth}{!}{$
A(G \lor H) = \begin{bmatrix} 
A(G) & J_{n_1, n_2} \\ 
J_{n_2, n_1} & A(H) 
\end{bmatrix}, \qquad
D(G \lor H) = \begin{bmatrix} 
D(G) + n_2 I_{n_1} & 0 \\ 
0 & D(H) + n_1 I_{n_2}
\end{bmatrix}.
$}
\end{align*}
Consequently,
\begin{equation}\label{eq:theo_union} 
L_\alpha(G \lor H) =
\begin{bmatrix}
L_\alpha(G) + \alpha n_2 I_{n_1} & (\alpha-1) J_{n_1, n_2} \\
(\alpha-1) J_{n_2, n_1} & L_\alpha(H) + \alpha n_1 I_{n_2}
\end{bmatrix},
\end{equation}
for $\alpha \in [0,1].$
Let $j_{n_1} \in \mathbb{R}^{n_1}$ and $j_{n_2} \in \mathbb{R}^{n_2}$ denote the all-ones vectors. The next results involve the  $L_\alpha(G \lor H)-$spectrum.

\begin{theorem}\label{theo:joinvet}
Let $x \in \mathbb{R}^{n_1}$ and $y \in \mathbb{R}^{n_2}$ be eigenvectors of $L_{\alpha}(G)$ e $L_{\alpha}(H),$ respectively, associated with the eigenvalues $\lambda$ and $\mu$,  such that $x$ and $y$ are orthogonal to the vectors $j_{n_1}$ and $j_{n_2},$ respectively. Then
\[
\{\alpha n_2 + \lambda,\; \alpha n_1 + \mu\} \subset \mathrm{Spec}(L_\alpha(G \lor H)),
\]
for $\alpha \in [0,1].$
\end{theorem}

\begin{proof}
Let $G$ and $H$ graphs of orders $n_1$ and $n_2$, respectively, and define
\[
X^\top = \begin{bmatrix} x & 0 \end{bmatrix}, \qquad
Y^\top = \begin{bmatrix} 0 & y \end{bmatrix}.
\]
From~\eqref{eq:theo_union} and the relations $J_{n_2, n_1}x = {\bf{0}}_{n_2,1}$ and $J_{n_1, n_2}y = {\bf{0}}_{n_1,1}$, it follows that
\[
L_\alpha(G \lor H)X = (\alpha n_2 + \lambda)X, \qquad
L_\alpha(G \lor H)Y = (\alpha n_1 + \mu)Y,
\]
which proves the result.
\end{proof}

\begin{theorem}\label{33}
Let $G$ be a $k$-regular graph of order $n_1$, and let $H$ be an $r$-regular graph of order $n_2$. For $\alpha \in [0,1]$,
\begin{equation*}
\resizebox{\textwidth}{!}{$
\mathrm{Spec}(L_\alpha(G \lor H)) =
\left\{
\lambda_1(M), \lambda_2(M),\;
\alpha (k+n_2) + (\alpha-1)\lambda_i(A(G)),\;
\alpha (r+n_1) + (\alpha-1)\lambda_j(A(H))
\right\},
$}
\end{equation*}
for $2 \le i \le n_1$ and $2 \le j \le n_2$, where
\[
M =
\begin{bmatrix}
(2k + n_2)\alpha - k & (\alpha-1)n_2 \\
(\alpha-1)n_1 & (2r + n_1)\alpha - r
\end{bmatrix}.
\]
\end{theorem}

\begin{proof}
Let $G$ be a $k$-regular graph of order $n_1$, and let $H$ be an $r$-regular graph of order $n_2$. Applying Theorem~\ref{equitable} to the matrix in~\eqref{eq:theo_union}, we obtain the quotient matrix $M$. From Theorem~\ref{quociente}, the eigenvalues of $M$ belong to $\mathrm{Spec}(L_\alpha(G \lor H))$. 

Let $\{j_{n_1}, x_2, \ldots, x_{n_1} \}$ be an orthogonal basis of eigenvectors of $A(G)$ such that $A(G)j_{n_1} = kj_{n_1}$ and $A(G)x_i = \lambda_i(A(G))x_i$, for all $2\leq i \leq n_1$. Define the vectors $\overline{x}_i = \begin{bmatrix} x_i \\ 0 \end{bmatrix}$, where the first $n_1$ entries of $\overline{x}_i$ coincide with those of $x_i$, and the remaining entries are zero, for all $2\leq i \leq n_1$. Then,
\begin{align*}
L_\alpha(G \lor H) \overline{x}_i&= \begin{bmatrix}
\alpha n_2 I_{n_1} + L_\alpha(G) & (\alpha-1) J_{n_1 \times n_2} \\
(\alpha-1) J_{n_2 \times n_1} & \alpha n_1 I_{n_2} + L_\alpha(H)
\end{bmatrix}
\begin{bmatrix}
x_i \\
0
\end{bmatrix} \\
&= \begin{bmatrix}
(\alpha (k + n_2) I_{n_1} + (\alpha-1)A(G))x_i \\
0
\end{bmatrix} \\
&= \begin{bmatrix}
\alpha (k + n_2) x_i + (\alpha-1)\lambda_i(A(G))x_i \\
0
\end{bmatrix} \\
&= \left(\alpha (k + n_2) + (\alpha-1)\lambda_i(A(G))\right)\overline{x}_i.
\end{align*}

Hence, $\alpha (k + n_2) + (\alpha-1)\lambda_i(A(G))$ is an eigenvalue of $L_\alpha(G \lor H)$, for $2\leq i \leq n_1$.

Applying the same argument to the graph $H$, which is $r$-regular, we obtain that $\alpha (r + n_1) + (\alpha-1)\lambda_j(A(H))$ is an eigenvalue of $L_\alpha(G \lor H)$, for $2\leq j \leq n_2$, and the result follows.

\end{proof}

For graphs $G$ and $H$ of orders $n_1$ and $n_2$, respectively, we have
\[
A(G \times H) = A(G) \otimes I_{n_2} + I_{n_1} \otimes A(H), \qquad
D(G \times H) = D(G) \otimes I_{n_2} + I_{n_1} \otimes D(H).
\]
Consequently,
\begin{equation}\label{eq:cartesian}
L_\alpha(G \times H) = L_\alpha(G) \otimes I_{n_2} + I_{n_1} \otimes L_\alpha(H).
\end{equation}

Theorem~\ref{pk} presents the $L_\alpha(G \times H)-$spectrum.

\begin{theorem}\label{pk}
Let $G$ and $H$ be graphs of orders $n_1$ and $n_2$, respectively, and let $\alpha \in [0,1]$. Then
\[
\mathrm{Spec}(L_\alpha(G \times H)) =
\left\{
\lambda_i(L_\alpha(G)) + \lambda_j(L_\alpha(H))
\right\},
\]
for $1 \le i \le n_1$ and $ 1 \le j \le n_2.$
\end{theorem}

\begin{proof}
Since $L_\alpha(G)$ and $L_\alpha(H)$ are symmetric, there exist orthogonal bases of $\mathbb{R}^{n_1}$ and $\mathbb{R}^{n_2}$ consisting of their eigenvectors. Let $X$ and $Y$ be orthogonal matrices whose columns are these eigenvectors, respectively. So
\[
X^T L_\alpha(G) X = \mathrm{diag}(\lambda_1(L_\alpha(G)),\dots,\lambda_{n_1}(L_\alpha(G))), \quad
Y^T L_\alpha(H) Y = \mathrm{diag}(\lambda_1(L_\alpha(H)),\dots,\lambda_{n_2}(L_\alpha(H))).
\]
Using~\eqref{eq:cartesian} and properties of the Kronecker product, we obtain
\[
(X \otimes Y)^T L_\alpha(G \times H) (X \otimes Y)
=
\mathrm{diag}(\lambda_i(L_\alpha(G)) + \lambda_j(L_\alpha(H))),
\]
for all $1 \le i \le n_1$ and $1 \le j \le n_2$, which proves the result.
\end{proof}

%\subsection{Direct product}
For the direct product, we have
\[
A(G \odot H) = A(G) \otimes A(H), \qquad
D(G \odot H) = D(G) \otimes D(H),
\]
and therefore
\begin{equation}\label{eq:direct}
L_\alpha(G \odot H)
=
\alpha \big(D(G) \otimes D(H)\big)
+
(\alpha - 1)\big(A(G) \otimes A(H)\big).
\end{equation}

The next theorems involve the $L_\alpha(G \odot  H)-$spectrum.

\begin{theorem}\label{teorema35}
Let $H$ be an $r$-regular graph of order $n_2$, and let $G$ be a connected graph of order $n_1$. Then
\[
\bigcup_{i=1}^{n_1} \{ r\,\lambda_i(L_\alpha(G)) \}
\subseteq
\mathrm{Spec}(L_\alpha(G \odot H)).
\]
for $\alpha \in [0,1].$
\end{theorem}

\begin{proof}
Let $x_i$ be an eigenvector of $L_\alpha(G)$ associated with $\lambda_i(L_\alpha(G))$, and let $j_{n_2}$ be the all-ones vector. Since $H$ is $r$-regular, we have
\[
D(H)j_{n_2} = r j_{n_2}, \qquad A(H)j_{n_2} = r j_{n_2}.
\]
Using~\eqref{eq:direct}, we have
\begin{align*}
L_\alpha(G \odot H)(x_i \otimes j_{n_2})
&=
\alpha\,(D(G)x_i \otimes D(H)j_{n_2})
+
(\alpha-1)\,(A(G)x_i \otimes A(H)j_{n_2}) \\
&=
r\Big(\alpha D(G)x_i + (\alpha-1)A(G)x_i\Big)\otimes j_{n_2} \\
&=
r\,L_\alpha(G)x_i \otimes j_{n_2} \\
&=
r\,\lambda_i(L_\alpha(G)) (x_i \otimes j_{n_2}),
\end{align*}
and the result follows.
\end{proof}

\begin{theorem}\label{teorema36}
Let $G$ be an $r_1$-regular graph of order $n_1$, and let $H$ be an $r_2$-regular graph of order $n_2$. For $\alpha \in [0,1],$
\[
\mathrm{Spec}(L_{\alpha}(G \odot H)) =
\bigcup_{i=1}^{n_{1}} \bigcup_{j=1}^{n_{2}}
\left\{
\alpha r_1 r_2 + (\alpha-1)\lambda_i(A(G))\lambda_j(A(H))
\right\}.
\]
\end{theorem}

\begin{proof}
Let $x_i$ and $y_j$ be eigenvectors of $A(G)$ and $A(H)$ corresponding to $\lambda_i(A(G))$ and $\lambda_j(A(H))$, respectively. Since $G$ and $H$ are regular, we also have
\[
D(G)x_i = r_1 x_i, \qquad D(H)y_j = r_2 y_j.
\]
Using~\eqref{eq:direct}, we obtain
\begin{align*}
L_\alpha(G \odot H)(x_i \otimes y_j)
&=
\alpha (D(G)x_i \otimes D(H)y_j)
+
(\alpha-1)(A(G)x_i \otimes A(H)y_j) \\
&=
\alpha r_1 r_2 (x_i \otimes y_j)
+
(\alpha-1)\lambda_i(A(G))\lambda_j(A(H))(x_i \otimes y_j),
\end{align*}
and the result follows.
\end{proof}

Since $G \times H$ and $G \odot H$ are edge-disjoint subgraphs of $G \otimes H$ and $E(G \times H)\cup E(G \odot H)=E(G\otimes H),$
it follows that
\[
A(G \otimes H) = A(G \times H) + A(G\odot H), \qquad 
D(G \otimes H) = D(G\times H) + D(G\odot H).
\]
Consequently,
\begin{equation}\label{eq:otimes}
L_\alpha(G \otimes H)
=
L_\alpha(G \times H) + L_\alpha(G\odot H).
\end{equation}

The next theorems involve the $L_\alpha(G \otimes  H)-$spectrum.

\begin{theorem}
Let $H$ be an $r$-regular graph of order $n_2$, and let $G$ be a connected graph of order $n_1$. For $\alpha \in [0,1]$, we have
\[
\left\{
(r+1)\lambda_i(L_\alpha(G)) + (2\alpha r - r)
\right\}
\subseteq
\mathrm{Spec}(L_\alpha(G \otimes H)),
\]
for $1 \le i \le n_1.$
\end{theorem}

\begin{proof}
Let $x_i$ be an eigenvector of $L_\alpha(G)$ associated with $\lambda_i(L_\alpha(G))$, and let $j_{n_2}$ be the all-ones vector. From~\eqref{eq:cartesian}, we have
\[
L_\alpha(G \times H)(x_i \otimes j_{n_2})
=
\lambda_i(L_\alpha(G)) (x_i \otimes j_{n_2})
+
L_\alpha(H)j_{n_2} \otimes x_i.
\]
Since $H$ is $r$-regular,
\[
L_\alpha(H)j_{n_2} = (2\alpha r - r)j_{n_2}.
\]
Moreover, by Theorem~\ref{teorema35},
\[
L_\alpha(G \odot H)(x_i \otimes j_{n_2})
=
r\,\lambda_i(L_\alpha(G))(x_i \otimes j_{n_2}).
\]
Combining these expressions with~\eqref{eq:otimes}, we obtain
\[
L_\alpha(G \otimes H)(x_i \otimes j_{n_2})
=
\big[(r+1)\lambda_i(L_\alpha(G)) + (2\alpha r - r)\big](x_i \otimes j_{n_2}),
\]
and the resul follows.
\end{proof}

\begin{theorem}
Let $G$ be an $r_1$-regular graph of order $n_1$ and $H$ an $r_2$-regular graph of order $n_2$. For $\alpha \in [0,1],$
\begin{equation*}
\begin{aligned}
\mathrm{Spec}(L_{\alpha}(G \otimes H)) 
= \bigcup_{i=1}^{n_{1}} \bigcup_{j=1}^{n_{2}} \Big\{ \,
& \alpha r_{1} r_{2} + (\alpha-1) \lambda_{i}(A(G)) \lambda_{j}(A(H)) \\[4pt]
& + \alpha (r_1 + r_2) + (\alpha-1)(\lambda_i(A(G)) + \lambda_j(A(H)))
\Big\},
\end{aligned}
\end{equation*}
for $1 \le i \le n_1$ and $1 \le j \le n_2.$
\end{theorem}

\begin{proof}
Let $x_i$ and $y_j$ be eigenvectors of $A(G)$ and $A(H)$ associated with $\lambda_i(A(G))$ and $\lambda_j(A(H))$, respectively. Since $G$ and $H$ are regular,
\[
D(G)x_i = r_1 x_i, \qquad D(H)y_j = r_2 y_j.
\]
Using~\eqref{eq:cartesian},~\eqref{eq:direct}, and~\eqref{eq:otimes}, a direct computation we have
\begin{align*}
L_\alpha(G \otimes H)(x_i \otimes y_j)
&=
\Big[
\alpha(r_1 r_2 + r_1 + r_2)
+ (\alpha-1)\big(\lambda_i(A(G))\lambda_j(A(H)) \\
&\qquad + \lambda_i(A(G)) + \lambda_j(A(H))\big)
\Big](x_i \otimes y_j).
\end{align*}
which completes the proof.
\end{proof}

The next result presents an expression for the characteristic polynomial of $L_\alpha$ for graphs obtained by coalescence between the vertices of two graphs. We denote by $L_\alpha(G)_u$ the principal sub-matrix of $L_\alpha(G)$ obtained by removing the row and column corresponding to vertex $u$, whose characteristic polynomial is denoted by $P_{L_\alpha(G)_u}(\lambda)$.

\begin{theorem}\label{thm:coalescence}
Let $G$ and $H$ be disjoint graphs, and let $G\cdot H$ be the graph obtained by identifying a vertex $u \in V(G)$ with a vertex $v \in V(H)$. Then
\[
P_{L_{\alpha}(G\cdot H)}(\lambda)
=
P_{L_{\alpha}(G)}(\lambda)P_{L_{\alpha}(H)_v}(\lambda)
+
P_{L_{\alpha}(G)_u}(\lambda)P_{L_{\alpha}(H)}(\lambda)
-
\lambda\,P_{L_{\alpha}(G)_u}(\lambda)P_{L_{\alpha}(H)_v}(\lambda),
\]
for $\alpha \in [0,1].$
\end{theorem}

\begin{proof}
Order the vertices of $G\cdot H$ so that the first vertex is the identified one, followed by the vertices of $G\setminus\{u\}$ and then those of $H\setminus\{v\}$, we have
\[
\lambda I - L_\alpha(G\cdot H)=
\begin{bmatrix}
\lambda-\alpha(d_G(u)+d_H(v)) & \nu^T & \rho^T\\
\nu & \lambda I - L_\alpha(G)_u & 0\\
\rho & 0 & \lambda I - L_\alpha(H)_v
\end{bmatrix},
\]
where $\nu$ and $\rho$ are column vectors whose nonzero entries, corresponding to the neighbors of $u$ in $G$ and $v$ in $H$, are equal to $1-\alpha$. Consider
\[
a=\lambda-\alpha(d_G(u)+d_H(v)), \qquad
X=\lambda I - L_\alpha(G)_u, \qquad
Y=\lambda I - L_\alpha(H)_v.
\]
Then
\[
P_{L_{\alpha}(G\cdot H)}(\lambda)
=
\det
\begin{bmatrix}
a & \nu^T & \rho^T\\
\nu & X & 0\\
\rho & 0 & Y
\end{bmatrix}.
\]

Applying Corollary~\ref{cor:schur-blockdiag}, we obtain
\begin{equation}\label{eq:coal-schur}
P_{L_{\alpha}(G\cdot H)}(\lambda)
=
\det(X)\det(Y)
\Bigl(
\lambda-\alpha(d_G(u)+d_H(v))
-\nu^TX^{-1}\nu
-\rho^TY^{-1}\rho
\Bigr).
\end{equation}

Now consider the decomposition of $\lambda I - L_\alpha(G)$ with respect to the vertex $u$:
\[
\lambda I - L_\alpha(G)=
\begin{bmatrix}
\lambda-\alpha d_G(u) & \nu^T\\
\nu & X
\end{bmatrix}.
\]
For all $\lambda$ such that $X$ is invertible, Proposition~\ref{prop:schur-scalar} yields
\[
P_{L_{\alpha}(G)}(\lambda)
=
\det(X)
\Bigl(
\lambda-\alpha d_G(u)-\nu^TX^{-1}\nu
\Bigr).
\]
Similarly,
\[
P_{L_{\alpha}(H)}(\lambda)
=
\det(Y)
\Bigl(
\lambda-\alpha d_H(v)-\rho^TY^{-1}\rho
\Bigr).
\]

Moreover,
\[
P_{L_{\alpha}(G)_u}(\lambda)=\det(X),
\qquad
P_{L_{\alpha}(H)_v}(\lambda)=\det(Y).
\]

Substituting these expressions, we obtain
\begin{align*}
&P_{L_{\alpha}(G)}(\lambda)P_{L_{\alpha}(H)_v}(\lambda)
+
P_{L_{\alpha}(G)_u}(\lambda)P_{L_{\alpha}(H)}(\lambda)
-
\lambda\,P_{L_{\alpha}(G)_u}(\lambda)P_{L_{\alpha}(H)_v}(\lambda)\\
&=
\det(X)\det(Y)
\Bigl(
\lambda-\alpha(d_G(u)+d_H(v))
-\nu^TX^{-1}\nu
-\rho^TY^{-1}\rho
\Bigr).
\end{align*}

Comparing the expression above with~\eqref{eq:coal-schur}, we conclude that
\[
P_{L_{\alpha}(G\cdot H)}(\lambda)
=
P_{L_{\alpha}(G)}(\lambda)P_{L_{\alpha}(H)_v}(\lambda)
+
P_{L_{\alpha}(G)_u}(\lambda)P_{L_{\alpha}(H)}(\lambda)
-
\lambda\,P_{L_{\alpha}(G)_u}(\lambda)P_{L_{\alpha}(H)_v}(\lambda).
\]

Finally, since both sides are polynomials in $\lambda$ and coincide for all $\lambda$ such that $X$ and $Y$ are invertible, they are identical by Proposition~\ref{prop:identidadepolinomial} and the proof is concluded.
\end{proof}

\subsection{Families of graphs}

In this subsection, we present the $L_{\alpha}-$spectrum of some families of graphs. Firstly, given a graph $ G$, consider  $S(G)$. With the convenient labeling  of $V(S(G))$ we have
\begin{align*}
A(S(G)) &= \begin{bmatrix} 
A(G) & A(G) \\ 
A(G) & 0 
\end{bmatrix}, \qquad
D(S(G)) = \begin{bmatrix} 
2D(G) & 0 \\ 
0 & D(G) 
\end{bmatrix},
\end{align*}
and
\begin{equation}\label{equacao6}
L_\alpha(S(G)) = \begin{bmatrix} 
2\alpha D(G) + (\alpha-1)A(G)  & (\alpha - 1) A(G) \\ 
(\alpha - 1) A(G) & \alpha D(G) 
\end{bmatrix}.
\end{equation}

Theorem~\ref{teorema39} presents the polynomial characteristic of $L_{\alpha}(S(G)).$

\begin{theorem}\label{teorema39}
Let $G$ be a $k$-regular graph of order $n$, and let $\lambda_i(A)$ denote the eigenvalues of $A(G)$, $1 \leq i \leq n$. For $\alpha \in [0,1],$ the characteristic polynomial of $L_\alpha(S(G))$ is given by
\[
\resizebox{\textwidth}{!}{$
P_{L_\alpha(S(G))}(x) = \prod_{i=1}^n \left(x^2 - 3\alpha kx + 2\alpha^2k^2 + (-(\alpha - 1)x + (\alpha - 1)\alpha k)\lambda_i(A) - (\alpha - 1)^2\lambda_i^2(A)\right).$}
\]
\end{theorem}

\begin{proof}
Let $A(G)=A$. Using~\eqref{equacao6} and the fact that $G$ is $k$-regular, we obtain
\[
L_\alpha(S(G)) = 
\begin{bmatrix}
2\alpha k I_n + (\alpha - 1) A & (\alpha - 1) A \\
(\alpha - 1) A & \alpha k I_n
\end{bmatrix}.
\]
From Theorem~\ref{determinante}, we have

\resizebox{\textwidth}{!}{$
\begin{aligned}
P_{L_\alpha(S(G))}(x) &= \det\left(((x - 2\alpha k)I_n - (\alpha - 1)A)((x - \alpha k)I_n) - (\alpha - 1)^2A^2\right) \\
&= \det\left((x^2 - 3\alpha kx + 2\alpha^2k^2)I + (-(\alpha - 1)x + (\alpha - 1)\alpha k)A - (\alpha - 1)^2A^2\right)
\end{aligned}
$}

As $\displaystyle \det(xI_{2n} - L_\alpha(S(G))) = \prod_{i=1}^n \lambda_i(xI_{2n} - L_\alpha(S(G))),$ we have that

\resizebox{\textwidth}{!}{$
\begin{aligned}
\lambda_i(xI_{2n} - L_\alpha(S(G))) &= \lambda_i\left((x^2 - 3\alpha kx + 2\alpha^2k^2)I_n + (-(\alpha - 1)x + (\alpha - 1)\alpha k)A - (\alpha - 1)^2A^2\right) \\
& =  x^2 - 3\alpha kx + 2\alpha k^2 + (-(\alpha-1)x+(\alpha-1)\alpha k)\lambda_i(A) - (\alpha-1)^2\lambda_i^2(A).
\end{aligned}
$}
So
\[
\resizebox{\textwidth}{!}{$
P_{L_\alpha(S(G))}(x) = \displaystyle \prod_{i=1}^{n} \left(x^2 - 3\alpha kx + 2\alpha^2k^2 + (-(\alpha - 1)x + (\alpha - 1)\alpha k)\lambda_i(A) - (\alpha - 1)^2\lambda_i^2(A)\right),
$}
\]
and the result follows.
\end{proof}

The next results involve the spectrum of the families, $K_{p}^{q},$ $H_n^l,$ $KK_n^l$ and $\Theta(c, s, \eta)$.

\begin{theorem}
Let $\alpha \in [0,1)$ and let $G \simeq K_{p}^{q}$ with $p \geq 3$. Then
\[
\mathrm{Spec}(L_{\alpha}(G))=\{\alpha^{(q-1)},\; (\alpha(p-2)+1)^{(p-2)},\; x_1,x_2,x_3\},
\]
where $x_1,x_2,x_3$ are the roots of
\[
\resizebox{\textwidth}{!}{$
\begin{aligned}
f(x) =\ & x^3 + (-3\alpha p - \alpha q + 4\alpha + p - 2)x^2 \\
& + (2\alpha^3p^2 - 6\alpha^2p + 4\alpha^2 - \alpha p^2 + 2\alpha^2pq - 4\alpha^2q + 5\alpha p - 4\alpha - \alpha pq + 4\alpha q - p - q + 1)x \\
& + (2\alpha^3pq - 3\alpha^3q - 5\alpha^2pq + 8\alpha^2q + 4\alpha pq - 7\alpha q - pq + 2q).
\end{aligned}
$}
\]
\end{theorem}

\begin{proof}
%Order the vertices so that
Using a convenient labeling of the vertices of $G$, we have
\[
L_\alpha(G) = \begin{bmatrix} 
B_{11} & (\alpha-1)J_{(p-1)\times 1} & 0_{(p-1)\times q} \\ 
(\alpha-1)J_{1\times (p-1)} & (p+q-1)\alpha & (\alpha-1)J_{1\times q} \\
0_{q\times (p-1)} & (\alpha-1)J_{q\times 1} & \alpha I_q
\end{bmatrix},
\]
where
\[
B_{11}=\alpha(p-1)I_{p-1}+(\alpha-1)(J_{p-1}-I_{p-1}).
\]
%The corresponding quotient matrix is
Applying Theorem~\ref{equitable} to the matrix above, we obtain the quotient matrix 
\[
M_{L_\alpha(G)} =
\begin{bmatrix} 
\alpha(p-1) + (\alpha-1)(p-2) & \alpha-1 & 0 \\ 
\alpha-1 & (p+q-1)\alpha & \alpha-1\\
0 & \alpha-1 & \alpha 
\end{bmatrix}
\]
whose characteristic polynomial is $f(x)$. From Theorem~\ref{quociente}, its eigenvalues belong to $\mathrm{Spec}(L_\alpha(G))$. Moreover, by Theorem~\ref{theo:twins}, the eigenvalue $\alpha$ has multiplicity $q-1$, and $\alpha(p-2)+1$ has multiplicity $p-2$. So the result follows.
\end{proof}

\begin{theorem}\label{hln}
Let $\alpha \in [0,1)$ and let $G \simeq H_n^l$, with $1 \leq l < n$. Then
\[
\mathrm{Spec}(L_\alpha(G)) =
\{(n\alpha)^{(l-1)},\; (\alpha(n-2)+1)^{(2n-2l-2)},\; (\alpha(n-2)+2)^{(l-1)},\; \theta_1, \theta_2, \theta_3, \theta_4 \},
\]
where
\[
\resizebox{\textwidth}{!}{$
\begin{aligned}
\theta_1 &= \frac{3\alpha n}{2} - \alpha - \frac{n}{2}
- \frac{\sqrt{8\alpha^2 l + \alpha^2 n^2 - 4\alpha^2 n + 4\alpha^2
- 12\alpha l - 2\alpha n^2 + 6\alpha n - 4\alpha
+ 4l + n^2 - 2n + 1}}{2}
+ \frac{1}{2}, \\[6pt]
\theta_2 &= \frac{3\alpha n}{2} - \alpha - \frac{n}{2}
+ \frac{\sqrt{8\alpha^2 l + \alpha^2 n^2 - 4\alpha^2 n + 4\alpha^2
- 12\alpha l - 2\alpha n^2 + 6\alpha n - 4\alpha
+ 4l + n^2 - 2n + 1}}{2}
+ \frac{1}{2}, \\[6pt]
\theta_3 &= \frac{3\alpha n}{2} - 2\alpha - \frac{n}{2}
- \frac{\sqrt{\alpha^2 n^2 + 4\alpha l - 2\alpha n^2
- 2\alpha n - 4l + n^2 + 2n + 1}}{2}
+ \frac{3}{2}, \\[6pt]
\theta_4 &= \frac{3\alpha n}{2} - 2\alpha - \frac{n}{2}
+ \frac{\sqrt{\alpha^2 n^2 + 4\alpha l - 2\alpha n^2
- 2\alpha n - 4l + n^2 + 2n + 1}}{2}
+ \frac{3}{2}.
\end{aligned}
$}
\]
\end{theorem}

\begin{proof}
Let $G \simeq H_n^l$, where $1 \leq l < n$. Using a convenient labeling of the vertices of $G$, the matrix $L_\alpha(G)$ can be written as
\[
L_\alpha(G)=
\begin{bmatrix}
B_{22} & (\alpha-1)J_{(n-l)\times l} & 0 & 0\\[4pt]
(\alpha-1)J_{l\times(n-l)} & B_{11} & (\alpha-1)I_l & 0\\[4pt]
0 & (\alpha-1)I_l & B_{11} & (\alpha-1)J_{l\times(n-l)}\\[4pt]
0 & 0 & (\alpha-1)J_{(n-l)\times l} & B_{22}
\end{bmatrix},
\]
where
\[
B_{11}=n\alpha I_l+(J_l-I_l)(\alpha-1),
\qquad
B_{22}=(n-1)\alpha I_{(n-l)}+(J_{(n-l)}-I_{(n-l)})(\alpha-1).
\]

Applying Theorem~\ref{equitable} in $L_\alpha(G)$, we obtain the quotient matrix
\[
M_{L_\alpha(G)}=
\begin{bmatrix}
b_{22} & l(\alpha-1) & 0 & 0\\[4pt]
(\alpha-1)(n-l) & b_{11} & \alpha-1 & 0\\[4pt]
0 & \alpha-1 & b_{11} & (\alpha-1)(n-l)\\[4pt]
0 & 0 & l(\alpha-1) & b_{22}
\end{bmatrix},
\]
where
\[
b_{11}=\alpha n+(\alpha-1)(l-1),
\qquad
b_{22}=\alpha(n-1)+(\alpha-1)(n-l-1).
\]

The eigenvalues of $M_{L_\alpha(G)}$ are precisely $\theta_1,\theta_2,\theta_3,\theta_4$ which belong to $\mathrm{Spec}(L_\alpha(G))$. Moreover, from Theorem~\ref{theo:twins}, $\alpha(n-2)+1$ is an eigenvalue of $L_\alpha(G)$ with multiplicity $2(n-l-1)$.

Let $e_i \in \mathbb{R}^{2n}$ denote the canonical basis vectors. Let $S=\{s_1,\dots,s_l\}$ be the set of vertices in the first copy of $K_n$ that are connected to the second copy, and let $S' = \{s_1',\dots,s_l'\}$ be the corresponding vertices in the second copy. For each $k=2,\dots,l$, define $u_k = e_{s_1}-e_{s_k} + e_{s_1'}-e_{s_k'}.$ 

So
$
L_\alpha(G)u_k = n\alpha\,u_k
$
and consequently $u_k$ is an eigenvector associated with the eigenvalue $n\alpha$. Since $\{u_2,\dots,u_l\}$ is linearly independent, the multiplicity of $n\alpha$ is at least $l-1$. Similarly, for each $k=2,\dots,l$, define $
v_k = e_{s_1}-e_{s_k}-e_{s_1'}+e_{s_k'}.$ Then
\[
L_\alpha(G)v_k
= \alpha D(G)v_k + (\alpha-1)A(G)v_k
= n\alpha\,v_k + (\alpha-1)(-2v_k)
= \bigl((n-2)\alpha+2\bigr)v_k.
\]
Thus, $v_k$ is an eigenvector associated with the eigenvalue $(n-2)\alpha+2$. Since $\{v_2,\dots,v_l\}$ is linearly independent, this eigenvalue has multiplicity at least $l-1$. So the result follows.
\end{proof}

\begin{theorem}
Let $\alpha \in [0,1)$ and let $G \simeq KK_n^l$ with $1 \leq l \leq n$. Then
\[
\mathrm{Spec}(L_{\alpha}(G)) = \{(\alpha(n-2)+1)^{(2n-l-3)},\ (\alpha(n-1)+ 1)^{(l-1)},\ x_1,x_2,x_3,x_4\},
\]
where $x_1,\dots,x_4$ are the roots of 
\[
f(x) = \sum_{k=0}^{4} C_k(\alpha, l, n)\, x^k,
\]
with coefficients
\begin{align*}
C_4 &= 1, \\
C_3 &= 2n^3 - 6n^2 + 6n - 2 - \alpha l + 7\alpha + 2l, \\
C_2 &= n^2 - 6n + 6 + \alpha(13n^2 - 31n + 18) + l(5\alpha n - 8\alpha + l - 2n + 2), \\
C_1 &= -2n^2 + 6n - 4 + \alpha(19n^2 - 38n + 21) + l(-34\alpha n + 34\alpha - 2l + 3) \\
    &\quad + 10\alpha n^3 - 52\alpha n^2 + 78\alpha n - 36\alpha + 26\alpha l n - 20\alpha l + l^2, \\
C_0 &= 4\alpha^4 l n^3 - 20\alpha^4 l n^2 + 32\alpha^4 l n - 16\alpha^4 l + 4\alpha^4 n^4 - 20\alpha^4 n^3 \\
    &\quad + 36\alpha^4 n^2 - 28\alpha^4 n + 8\alpha^4 - 4\alpha^3 l^2 n + 6\alpha^3 l^2 - 4\alpha^3 l n^3 \\
    &\quad + 34\alpha^3 l n^2 - 70\alpha^3 l n + 40\alpha^3 l - 4\alpha^3 n^4 + 28\alpha^3 n^3 \\
    &\quad - 64\alpha^3 n^2 + 60\alpha^3 n - 20\alpha^3 + 8\alpha^2 l^2 n - 13\alpha^2 l^2 + \alpha^2 l n^3 \\
    &\quad - 22\alpha^2 l n^2 + 57\alpha^2 l n - 36\alpha^2 l + \alpha^2 n^4 - 13\alpha^2 n^3 \\
    &\quad + 41\alpha^2 n^2 - 47\alpha^2 n + 18\alpha^2 - 5\alpha l^2 n + 9\alpha l^2 + 7\alpha l n^2 \\
    &\quad - 21\alpha l n + 14\alpha l + 2\alpha n^3 - 11\alpha n^2 + 16\alpha n - 7\alpha + l^2 n - 2l^2 \\
    &\quad - l n^2 + 3l n - 2l + n^2 - 2n + 1.
\end{align*}
\end{theorem}

\begin{proof}
%With a suitable labeling of the vertices, we can write
Using a convenient labeling of the vertices of $G$, we can write
\begin{align*}
L_\alpha(G) &= \begin{bmatrix} 
(n+l-1)\alpha & (\alpha-1)J_{1\times(n-1)} & (\alpha-1)J_{1\times l} & 0\\ 
(\alpha-1)J_{(n-1)\times 1} & B_{22} & 0 & 0 \\
(\alpha-1)J_{l\times 1} & 0 & B_{33} & (\alpha-1)J_{l\times(n-l)} \\
0 & 0 & (\alpha-1)J_{(n-l)\times l} & B_{44}
\end{bmatrix},
\end{align*}
where
\[
B_{22}=(n-1)\alpha I_{n-1}+(\alpha-1)(J_{n-1}-I_{n-1}),\quad
B_{33}=n\alpha I_l+(\alpha-1)(J_l-I_l),\]
\[
B_{44}=(n-1)\alpha I_{n-l}+(\alpha-1)(J_{n-l}-I_{n-l}).
\]
Applying Theorem~\ref{equitable} in $L_\alpha(G)$, the quotient matrix $M_{L_\alpha(G)}$ has as the characteristic polynomial $f(x)$. By Theorem~\ref{quociente}, its eigenvalues belong to $\mathrm{Spec}(L_\alpha(G))$. Furthermore, by Theorem~\ref{theo:twins}, the eigenvalue $\alpha(n-2)+1$ has multiplicity $2n-l-3$, and $\alpha(n-1)+1$ has multiplicity $l-1$. So the result follows.
\end{proof}

\begin{theorem}
Let $\alpha \in [0,1]$ and let $G \simeq \Theta(c, s, \eta) \simeq K_c \lor \eta K_s$. Then
\begin{multline*}
\mathrm{Spec}(L_\alpha(G)) = 
\Big\{\lambda_1(M), \lambda_2(M), 
\big(\alpha (c-1)+\eta s + (\alpha-1)(-1)\big)^{(c-1)}, \\
\big(\alpha (s-1+ c) + (\alpha-1)(s-1)\big)^{(\eta-1)}, \;
\big(\alpha (s-1+ c) + (\alpha-1)(-1)\big)^{(\eta s- \eta)}  \Big\},
\end{multline*}
where $M$ is as in Theorem~\ref{33}, with $G\simeq K_c$, $H\simeq \eta K_s$, $k=c-1$, $n_1=c$, $r=s-1$, and $n_2=\eta s$.
\end{theorem}

\begin{proof}
The result follows by combining Theorems~\ref{complete} and~\ref{33} with the indicated parameters.
\end{proof}

\section{Conclusion}\label{sec4}

In this paper, we investigate the spectrum of the family $L_\alpha$ under several graph operations, including union, join, products, and splitting. We provide explicit spectral descriptions of $L_\alpha$ for each of these operations, highlighting the interplay between the structure of the graphs and their spectral properties. In addition, we establish a formula for the characteristic polynomial of $L_\alpha$ under the coalescence operation. Furthermore, we derive explicit expressions for the spectra of $L_\alpha$ for several important graph families, including pineapple graphs, $H_n^l$, $KK_n^l$, and core--satellite graphs, the latter naturally arising from combinations of union and join operations. As a direction for future research, we propose extending the spectral analysis of $L_\alpha$ to broader classes of graphs and investigating the ordering of its eigenvalues and their connections to structural graph properties.

\section{Acknowledgment}

This study was partially funded by FAPERJ - Funda\-ção Carlos Chagas Filho de Amparo à Pesquisa do Estado do Rio de Janeiro, Process SEI 260003/001228/2023, CNPq-Conselho Nacional de Desenvolvimento Científico e Tecnológico, Grant 405552/2023-8 and Programa Institucional de Bolsas de Iniciação Científica - PIBIC/CNPq.

%Bibliography
\bibliographystyle{plainnat}
\bibliography{references}  

\end{document}